\documentclass[a4paper,twoside]{article}      
\usepackage{amsmath,amssymb,amsfonts,amsthm,amscd}
\usepackage{graphicx}
\usepackage{a4wide}
\usepackage{enumerate}
\usepackage{color}
\usepackage{mathrsfs}
\usepackage{authblk}
\def\R{\mathbb{R}}
\newtheorem{theorem}{Theorem}[section]

\newtheorem{lemma}[theorem]{Lemma}

\theoremstyle{definition}
\newtheorem{remark}[theorem]{Remark}

\def\N{\mathcal{N}}

\def\div{\mathop{\mathrm{div}}\nolimits}

\usepackage{hyperref}
\hypersetup{
    bookmarks=true,         
    unicode=false,          
    pdftoolbar=true,        
    pdfmenubar=true,        
    pdffitwindow=false,     
    pdfstartview={FitH},    
    pdftitle={My title},    
    pdfauthor={Author},     
    pdfsubject={Subject},   
    pdfcreator={Creator},   
    pdfproducer={Producer}, 
    pdfkeywords={keywords}, 
    pdfnewwindow=true,      
    colorlinks=true,       
    linkcolor=blue,          
    citecolor=blue,        
    filecolor=magenta,      
    urlcolor=cyan           
}

\newcommand{\AM}{\textcolor{black}}

\title{Model reduction of Brownian oscillators: quantification of errors and long-time behaviour}
\author[1]{Matteo Colangeli\thanks{matteo.colangeli1@univaq.it}}
\author[2]{Manh Hong Duong\thanks{h.duong@bham.ac.uk}}
\author[3]{Adrian Muntean\thanks{adrian.muntean@kau.se}}
\affil[1]{Department of Information Engineering, Computer Science and Mathematics,
University of L'Aquila, Italy.}
\affil[2]{School of Mathematics,
University of Birmingham,
UK.}
\affil[3]{Department of Mathematics and Computer Science \& Centre for Societal Risk Research (CSR), Karlstad University, Sweden.}

\begin{document}
\maketitle
\begin{abstract}
A procedure for model reduction of stochastic ordinary differential equations with additive noise was recently introduced in \cite{CDM22}, based on the Invariant Manifold method and on the Fluctuation-Dissipation relation. A general question thus arises as to whether one can rigorously quantify the error entailed by the use of the reduced dynamics in place of the original one. In this work we provide explicit formulae and estimates of the error in terms of the Wasserstein distance, both in the presence or in the absence of a sharp time-scale separation between the variables to be retained or eliminated from the description, as well as in the long-time behaviour.  

\AM{{\bf Keywords:} Model reduction, Wasserstein distance, error estimates, coupled Brownian oscillators, invariant manifold, Fluctuation-Dissipation relation.}
\end{abstract}
\section{Introduction}
\label{sec:sec1}

The notion of scale separation is largely invoked in multiscale modelling and homogeneization methods \AM{(including model reduction and operator splitting techniques)} \cite{givon2004extracting,pavliotis2008multiscale}, and has also found far-reaching applications in different areas of science \AM{and engineering}, \textit{e.g.} in climate dynamics \cite{Ghil20},  biochemical systems \cite{Schnell17}, chemical reaction networks \cite{Kang2013}, \AM{smoldering combustion \cite{Ijioma2014}, and so on}. 
A neat illustration of this \AM{notion} can be traced
in the preface of Haken's seminal book on Synergetics \cite{Haken2004}, where the author writes: ``In large classes of systems that are originally described by \textit{many} variables, the behavior of a system is described and determined by only \textit{few} variables, the \textit{order parameters}. They fix the behavior of the individual parts via the \textit{slaving principle}''. A physical rationale behind the slaving principle amounts to the assumption of decomposition of motions: there exists a short time-scale during which the slow variable does not change significantly, while the fast variable rapidly settles on a value determined by the slow one. The evolution of the latter, in turn, takes place on a much longer scale. 
A specific form of such principle is realized through the method of adiabatic elimination of fast variables, which underlies the derivation of the Smoluchowski equation from the underdamped Langevin equation. A sharp distinction between slow and fast variables is also a prerequisite for application of the Mori-Zwanzing method \cite{Zwanzig} in the derivation of reduced equations from higher dimensional stochastic dynamics, where the Markovian structure of the original process is preserved in the reduced description by stipulating a perfect time-scale separation.
The same guiding principle underpins, in kinetic theory, the Grad moment method \cite{Grad49,colan07b},  and has also been exploited in the derivation of linear hydrodynamics from the Boltzmann equation using the framework of the Invariant Manifold \cite{GorKar05,colan09}. 
The latter method has also been exploited in \cite{CDM22} to characterize the deterministic component of the contracted description in a system of two coupled (underdamped) Brownian harmonic oscillators. The structure of the noise term of the Markovian reduced dynamics, in turn, was determined via the Fluctuation-Dissipation relation. A general question, then, concerns the derivation of a quantitative estimate of the error stemming from the use of the reduced dynamics in place of the original one.
A first attempt, in this direction, was proposed in \cite{CM22}, and it was based on the study of the equilibrium correlation functions in the reduced and in the original processes. A uniform-in-time type of convergence of the correlations evaluated in the two processes was proven to hold in the so-called overdamped limit, where the friction parameter diverges.

In this work we take a step further, and compute explicitly the Wasserstein distance between the laws of the original and reduced processes. This paves the way to explicitly quantify the error inherent to the contracted description.
We focus on two classical models thoroughly studied in statistical physics and molecular dynamics, namely the underdamped Brownian harmonic oscillator and a system of two coupled overdamped Brownian harmonic oscillators. In the more traditional approach based on the slow-fast decomposition of motions, a reduced description can be achieved by passing the parameter to a certain limit, thus establishing a perfect time-scale separation, see e.g. \cite{Zwanzig,Ghil}. In the present work, instead, we derive the reduced dynamics in a regime characterized by a finite time-scale separation, which is controlled, in the two considered models, by either the friction parameter or the coupling parameter. We show that the reduced and original dynamics are exponentially close at any time, and they coincide if we pass the parameter to the corresponding limit. We also prove that the two dynamics have the same equilibrium measure and, furthermore, they exponentially converge to the equilibrium measure with the same rate. This notable property is a direct consequence of the proposed reduction scheme, in particular of the selection of solutions to the invariance equation obtained from the Invariant Manifold method. As a consequence of this, the spectrum of the reduced drift matrix is a subset of the spectrum of the original drift matrix. The models and precise statements of the results are presented in Section \ref{sec:sec3} and Section \ref{sec:sec4}.

The work is structured as follows. In Sec. \ref{sec:sec2} we review the definition of the Wasserstein distance between two probability measures and introduce the basic notation used throughout the manuscript. In Sec. \ref{sec:sec3} we compute our error estimate based on the Wasserstein distance for a Brownian harmonic oscillator, for which the laws of the original and the contracted descriptions are analytically known. In Sec. \ref{sec:sec4} we apply our method to a slightly more involved model, constituted by a pair of coupled overdamped Brownian harmonic oscillators. Conclusions and a final outlook are finally drawn in Sec. \ref{sec:sec5}.

\section{Preliminaries}
\label{sec:sec2}

In this Section we introduce the Wasserstein distance between two probability measures and also fix the notation used \AM{throughout} the manuscript.

\subsection{Wasserstein distance}
In this section we recall the definition of the Wasserstein distance between two probability measures and its explicit formula when the two probability measures are Gaussian distributions. The Wasserstein metric plays an central role in many research fields such as optimal transport, partial differential equations and data science. For a detailed account of the topics, we refer the reader to Villani's monograph \cite{Villani2003}.

Let $P_2(\R^d)$ be the space of probability measures $\mu$ on $\R^d$ with finite second moment, namely
$$
\int_{\R^d}|x|^2\mu(dx)<\infty.
$$
Let $\mu$ and $\nu$ be two probability measures belonging to $P_2(\R^d)$. The $L^2$-Wasserstein distance, $W_2(\mu,\nu)$, between $\mu$ and $\nu$ is defined via
\begin{equation}
\label{eq: W2}
W^2_2(\mu,\nu):=\inf_{\gamma\in \Gamma(\mu,\nu)}\int_{\R^d\times\R^d}|x-y|^2\,\gamma(dx,dy),
\end{equation}
where $\Gamma(\mu,\nu)$ denotes the set of all couplings between $\mu$ and $\nu$, i.e., the set of all probability measures on $\R^d\times \R^d$ having $\mu$ and $\nu$ as the first and the second marginals respectively. More precisely,
$$
\Gamma(\mu,\nu):=\{\gamma\in P(\R^d\times \R^d): \gamma(A\times \R^d)=\mu(A)~\text{and}~ \gamma(\R^d\times A)=\nu(A)\},
$$
for all Borel measurable sets $A\subset\R^d$.

In particular, the Wasserstein distance between two Gaussian measures can be computed explicitly in terms of the means and covariance matrices  \cite{GivensShortt1984}, see also e.g., \cite{Takatsu2012}
\begin{equation}
\label{eq: W2-Gaussians}
W_2(\N(u,U),\N(v,V))^2=|u-v|^2+\mathrm{tr}U+\mathrm{tr}V-2\mathrm{tr}\sqrt{V^\frac{1}{2}U V^\frac{1}{2}},
\end{equation}
where $u, v$ are the means and $U, V$ are the covariance matrices.
In a one dimensional space, the above formula reduces to
\begin{equation}
W_2(\N(u_1,\sigma_1^2),\N(u_2,\sigma_2^2)^2=(u_1-u_2)^2+(\sigma_1-\sigma_2)^2.
\end{equation}
\subsection{Linear drift-diffusion equations}
\label{sec: linear SDE}
We recall here a well-known result \AM{concerning} the explicit solution of a general linear drift-diffusion where the initial data is a Gaussian distribution. In the subsequent sections, we will apply this result to our models of (coupled) Brownian oscillators.

\AM{To set the stage, we }consider the following general linear drift-diffusion equation
\begin{equation}
\label{eq: linear diffusion}
    \partial_t\rho=-\div(Cx\rho)+\div(D\nabla\rho),\quad \rho(0)=\rho_0.
\end{equation}
In the above equation, the unknown is a probability measure  $\rho=\rho(t,x)$ with $(t,x)\in (0,\infty)\times \mathbb{R}^d$; $C$ and $D$ are two constant matrices of order $d$ representing the drift and diffusion matrices; the initial data $\rho_0$ is a probability measure on $\mathbb{R}^d$.

The following lemma provides the explicit formula for the solution of \eqref{eq: linear diffusion} when the initial data is a Gaussian distribution, see for instance \cite{gomes2018mean}.
\begin{lemma}
\label{lem: Gaussian sol}
Suppose the initial data is a Gaussian, $\rho_0\sim \mathcal{N}(\mu(0),\Sigma(0))$, then the solution to \eqref{eq: linear diffusion} is given by 
\begin{equation}
\label{eq: MKEGaussian}
\rho(t,x)=\frac{1}{\sqrt{(2\pi)^d\det\Sigma(t)}}\exp\Big[-\frac{1}{2}(x-\mu(t))^T \Sigma^{-1}(t)(x-\mu(t))\Big]  
\end{equation}
where $\mu(t)$ and $\Sigma(t)$ are given by
\begin{equation}
\label{eq: MKE Gaussian mean-variance}
\mu(t):=e^{tC}\mu(0),\quad \Sigma(t):=e^{t C}\Sigma(0)e^{t C^T}+2\int_0^t e^{s C}De^{s C^T}\,ds.
\end{equation}
Under suitable conditions on $C$ and $K$, we have $\mu(t)\rightarrow 0$ and $\Sigma(t)\rightarrow \Sigma_\infty$ where
$$
\Sigma_\infty:=2\int_0^\infty e^{sC}De^{sC^T}\,ds.
$$
Note that $\Sigma_\infty$ satisfies the so-called Lyapunov equation
$$
2D=C\Sigma_\infty+\Sigma_\infty C^T.
$$
\end{lemma}
\subsection{Exponential of a $2\times 2$ matrix}
Lemma \ref{lem: Gaussian sol} provides \AM{the explicit form of the unique} solution to the linear drift-diffusion equation \eqref{eq: linear diffusion} when the initial data is a Gaussian. However, in general the formula \eqref{eq: MKE Gaussian mean-variance} is analytically hard to compute since it involves exponential of matrices. The following lemma provides an explicit formula for the exponential of a $2\times 2$ matrix, which will be used in the subsequent analysis.
\begin{lemma}
\label{lem: exponential matrix}
\AM{Let $a,b,c,d \in\mathbb{R}$ be taken arbitrarily with $a^2+b^2+c^2+d^2>0$.} The following identity holds
\begin{equation}\label{magic}
    \exp\begin{pmatrix}
        a&b\\c &d
    \end{pmatrix}=\frac{1}{\Delta}\begin{pmatrix}
        m_{11}&m_{12}\\
        m_{21}& m_{22}
    \end{pmatrix},
\end{equation}
where $\Delta:=\sqrt{(a-d)^2+4bc}$ and
\begin{align*}
    m_{11}&:=e^{(a+d)/2}\Big[\Delta \cosh{\frac{1}{2}\Delta} +(a-d)\sinh{\frac{1}{2}\Delta}\Big],\\
    m_{12}&:=2be^{(a+d)/2}\sinh{\frac{1}{2}\Delta},
    \\ m_{21}&:=2 c e^{(a+d)/2}\sinh{\frac{1}{2}\Delta},
    \\ m_{22}&:=e^{(a+d)/2}\Big[\Delta \cosh{\frac{1}{2}\Delta} +(d-a)\sinh{\frac{1}{2}\Delta}\Big].
\end{align*}
\end{lemma}

\begin{proof}\AM{We refer the reader to \cite{bernstein1993} for a justification of the formula \eqref{magic}.}
\end{proof}

\section{Model reduction of a Brownian oscillator}
\label{sec:sec3}

\AM{To start off the discussion, we begin with the investigation of a } simple model of an underdamped Brownian oscillator considered in \cite{CM22}, which is amenable to an explicit analytical solution.
The original dynamics reads as follows:
\begin{align*}
dx(t)&=v(t)\,dt\\
d v(t)&=-\omega^2 x(t)\,dt-\gamma v(t)\,dt+\sqrt{2\gamma\beta^{-1}}\,dW(t),\\
(x(0),v(0))&=(x_0,v_0)
\end{align*}
Exploiting the Invariant Manifold method and the Fluctuation-Dissipation relation (for a short summary of the method, see Section \ref{sec:sec4} below, where the same reduction procedure is applied to a system of coupled overdamped Brownian harmonic oscillators), the reduced dynamics attains the form: 
$$
d \bar{x}(t)=-\alpha \bar{x}(t)\,dt+\sqrt{2 D_r}\,dW(t),\quad \bar{x}(0)=x_0 ,
$$
where 
$$
\alpha=\frac{\gamma-\sqrt{\gamma^2-4\omega^2}}{2},\quad D_r=\frac{\alpha}{\omega^2\beta}.
$$ \AM{The reader is referred to \cite{CM22} to see the details of the calculations}. 
The main result of this section is the following theorem.
\begin{theorem} 
\begin{enumerate}[(i)]\
\item (exact solutions of the original and the reduced dynamics) $\mu_t$ and $\bar{\mu}_t$ are Gaussian measures 
\begin{equation}
\mu_t=\mathcal{N}(m(t),\sigma(t)),\quad \bar{\mu}_t=\mathcal{N}(\bar{m}_t,\bar{\sigma}(t)),
\end{equation}
where
\begin{align*}
m(t)&=\frac{\lambda_1 e^{-\lambda_2 t}-\lambda_2 e^{-\lambda_1 t}}{\lambda_1-\lambda_2}x_0+\frac{e^{-\lambda_2 t}-e^{-\lambda_1 t}}{\lambda_1-\lambda_2}v_0,\\
\sigma(t)&=\frac{\gamma \beta^{-1}}{(\lambda_1-\lambda_2)^2}\Big[\frac{\lambda_1+\lambda_2}{\lambda_1\lambda_2}+\frac{4}{\lambda_1+\lambda_2}(e^{-(\lambda_1+\lambda_2)t}-1)-\frac{1}{\lambda_1}e^{-2\lambda_1 t}-\frac{1}{\lambda_2}e^{-2\lambda_2 t}\Big],\\
\bar{m}(t)&=e^{-\lambda_2 t}\bar{x}_0,\\
\bar{\sigma}(t)&=\frac{1}{\omega^2 \beta}(1-e^{-2\lambda_2 t})
\end{align*}
where
\begin{equation}
\lambda_1=\frac{\gamma+\sqrt{\gamma^2-4\omega^2}}{2},\quad \lambda_2=\frac{\gamma-\sqrt{\gamma^2-4\omega^2}}{2}=\frac{2\omega^2}{\gamma+\sqrt{\gamma^2-4\omega^2}}.
\end{equation}
\item (Exact Wasserstein distance between the laws of the original and reduced dynamics) The Wasserstein distance between $\mu_t$ and $\bar{\mu}_t$ can be computed explicitly via
\begin{equation}
W_2^2(\mu_t,\bar{\mu}_t)=(m(t)-\bar{m}(t))^2+\Big(\sqrt{\sigma_{xx}(t)}-\sqrt{\bar{\sigma}(t)}\Big)^2.
\end{equation}
\item (explicit rate of convergence in the high-friction limit) It holds that
\begin{equation}
\label{eq: high-friction}
W_2^2(\mu_t,\bar{\mu}_t)\leq \frac{4}{\gamma^2-4\omega^2}\Big[(\omega |x_0|+|v_0|)^2+\frac{4}{\beta}\Big]\quad\forall t>0.
\end{equation}
As a consequence,
$$
\lim_{\gamma\rightarrow +\infty}W_2^2(\mu_t,\bar{\mu}_t)=0.
$$
Note that \eqref{eq: high-friction} is a much stronger statement providing an explicit rate of convergence.
\item (Common rates of convergence to equilibrium) There exists a constant $C>0$, which can be found explicitly, such that
$$
W_2(\mu_t,\mu_\infty),~ W_2(\bar{\mu}_t,\bar{\mu}_\infty)\leq C e^{-\lambda_2 t},
$$
where 
$$
\mu_\infty=\bar{\mu}_\infty=\mathcal{N}\Big(0, \frac{1}{\beta\omega^2}\Big).
$$
This result shows that the original dynamics and the reduced one not only share the same equilibrium, they have the same rates of convergence to equilibrium in the Wasserstein distance.
\item (long-time behaviour) It holds that
\begin{equation}
\label{eq: longtime}
W_2^2(\mu_t,\bar{\mu}_t)\leq \Big[\frac{\omega|x_0|+|v_0|}{\sqrt{\gamma^2-4\omega^2}}+\frac{10}{\beta(\gamma^2-4\omega^2)}\Big]e^{-\lambda_2 t}.
\end{equation}
As a consequence of this, we also have
$$
\lim_{t\rightarrow +\infty}W_2(\mu_t,\bar{\mu}_t)=0,
$$
which is already obtained in the previous part. Estimate \eqref{eq: longtime} is a stronger statement, showing that the two dynamics are exponentially close at any time $t>0$.
\item Suppose that the initial data $x_0$ is randomly distributed according to an even probability measure $\rho_0\in L^1(\mathbb{R})$ then the estimates in parts $(iii)$ and $(iv)$ still hold true.
\end{enumerate}
\end{theorem}
\begin{proof}\
$(i)$. The law $\rho_t$ of $z(t)=\begin{pmatrix}x(t)\\v(t)\end{pmatrix}$ satisfies the kinetic Fokker Planck equation
$$
\partial_t\rho_t=\mathscr{L}^*\rho_t, \quad \rho\vert_{t=0}=\delta_{(x_0,v_0)},
$$
where $\mathscr{L}^*\rho:=-v\partial_x\rho+\omega^2 x\partial_v\rho+\gamma \big[\partial_{v}(v \rho)+\beta^{-1}\partial^2_{vv}\rho\big]$.

According to [Risken, Section 10.2] $\mu_t$ is a bivariate Gaussian measure with mean $M(t)\in \mathbb{R}^2$ and covariane matrix $\Sigma(t)\in\mathbb{R}^{2\times 2}$. \AM{They are $t$ dependent objects} given by
$$
M(t)=\begin{pmatrix}
m_x(x)\\ m_v(t)
\end{pmatrix}, \quad \Sigma^{-1}(t)=\begin{pmatrix}
[\sigma_{xx}(t)]^{-1}&[\sigma_{xv}(t)]^{-1}\\
[\sigma_{vx}(t)]^{-1}&[\sigma_{vv}(t)]^{-1}
\end{pmatrix},
$$
where 
\begin{align*}
m_x(t)&=\frac{\lambda_1 e^{-\lambda_2 t}-\lambda_2 e^{-\lambda_1 t}}{\lambda_1-\lambda_2}x_0+\frac{e^{-\lambda_2 t}-e^{-\lambda_1 t}}{\lambda_1-\lambda_2}v_0,\\
m_v(t)&=\omega^2 \frac{e^{-\lambda_1 t}-e^{-\lambda_2 t}}{\lambda_1-\lambda_2} x_0+\frac{\lambda_1 e^{-\lambda_1 t}-\lambda_2 e^{-\lambda_2 t}}{\lambda_1-\lambda_2}v_0,\\
\sigma_{xx}(t)&=\frac{\gamma \beta^{-1}}{(\lambda_1-\lambda_2)^2}\Big[\frac{\lambda_1+\lambda_2}{\lambda_1\lambda_2}+\frac{4}{\lambda_1+\lambda_2}(e^{-(\lambda_1+\lambda_2)t}-1)-\frac{1}{\lambda_1}e^{-2\lambda_1 t}-\frac{1}{\lambda_2}e^{-2\lambda_2 t}\Big],\\
\sigma_{xv}(t)&=\frac{\gamma \beta^{-1}}{(\lambda_1-\lambda_2)^2}(e^{-\lambda_1 t}-e^{-\lambda_2 t})^2,\\
\sigma_{vv}(t)&=\frac{\gamma \beta^{-1}}{(\lambda_1-\lambda_2)^2}\Big[\lambda_1+\lambda_2+\frac{4\lambda_1\lambda_2}{\lambda_1+\lambda_2}(e^{-(\lambda_1+\lambda_2)t}-1)-\lambda	_1 e^{-2\lambda_1 t}-\lambda_2 e^{-2\lambda_2 t}\Big],
\end{align*} 
where
\begin{equation}
\label{eq: lambdas}
\lambda_{1}=\frac{\gamma+\sqrt{\gamma^2-4\omega^2}}{2},\quad \lambda_{2}=\frac{\gamma-\sqrt{\gamma^2-4\omega^2}}{2},\quad\text{thus}\quad \lambda_1+\lambda_2=\gamma,\quad \lambda_1\lambda_2=\omega^2,\quad \lambda_1-\lambda_2=\sqrt{\gamma^2-4\omega^2}.
\end{equation}
Note that, since in the overdamped regime $\gamma\geq 2\omega$, we have
$$
\lambda_2=\frac{\gamma- \sqrt{\gamma^2-4\omega^2}}{2}=\frac{4\omega^2}{2(\gamma+\sqrt{\gamma^2-4\omega^2})}\leq \frac{4\omega^2}{4\omega}=\omega.
$$
Since $z(t)$ is a bivariate Gaussian, it follows that the law of $x(t)$, which is the first marginal of $z(t)$, is a univariate Gaussian measure, $\mu_t=\mathcal{N}(m(t),\sigma(t))$, with mean $m(t)=m_x(t)$ and variance $\sigma(t)=\sigma_{xx}(t)$, where $m_x(t)$ and $\sigma_{xx}(t)$ are defined above. Using \eqref{eq: lambdas} we can re-write $m(t)$ and $\sigma(t)$ as follows
\begin{align}
\label{eq: sigmaxx}
m(t)&=e^{-\lambda_2 t}x_0+\frac{e^{-\lambda_2 t}-e^{-\lambda_1 t}}{\lambda_1-\lambda_2}(\lambda_2 x_0+ v_0),\\
\sigma(t)&=\frac{\gamma\beta^{-1}}{(\gamma^2-4\omega^2)}\Big[\frac{\gamma}{\omega^2}+\frac{4}{\gamma}(e^{-\gamma t}-1)-\frac{1}{\lambda_1}e^{-2\lambda_1 t}-\frac{1}{\lambda_2}e^{-2\lambda_2 t}\Big]
\\&=\frac{1}{\beta \omega^2\,[1-4(\omega/\gamma)^2]}\big(1-e^{-2\lambda_2 t}\big)+\frac{\gamma\beta^{-1}}{(\gamma^2-4\omega^2)}\Big[\frac{4}{\gamma}(e^{-\gamma t}-1)-\frac{e^{-2\lambda_1 t}-e^{-2\lambda_2 t}}{\lambda_1}\Big],
\end{align}
where in the last equality we have used the following equality
\begin{align*}
\frac{1}{\lambda_1}e^{-2\lambda_1 t}+\frac{1}{\lambda_2}e^{-2\lambda_2 t}&=\frac{(\lambda_1+\lambda_2)e^{-2\lambda_2 t}}{\lambda_1 \lambda_2}+\frac{(e^{-2\lambda_1 t}-e^{-2\lambda_2 t})}{\lambda_1}
\\&=\frac{\gamma e^{-2\lambda_2 t}}{\omega^2}+\frac{(e^{-2\lambda_1 t}-e^{-2\lambda_2t})}{\lambda_1}
\end{align*}
The reduced dynamics is an Ornstein-Uhlenbeck process, therefore its law is a Gaussian measure, $\bar{\mu}_t=\mathcal{N}(\bar{m}(t),\bar{\sigma}^2(t))$, with mean
\begin{equation}
\bar{m}(t)=e^{-\alpha t} x_0=e^{-\lambda_2 t}x_0,
\end{equation}
and variance 
\begin{equation}
\bar{\sigma}(t)=\frac{D_r}{\alpha}(1-e^{-2\alpha t})=\frac{1}{\omega^2 \beta}(1-e^{-2\lambda_2 t}).
\end{equation}
$(ii)$ Using the general explicit formula for the Wasserstein distance between two univariate Gaussian measures, we obtain the Wasserstein distance between the original dynamics and the reduced dynamics, $W^2_2(\mu_t,\bar{\mu}_t)$, as follows
\begin{equation}
W_2^2(\mu_t,\bar{\mu}_t)^2=\Big(m_x(t)-\bar{m}(t)\Big)^2+\Big(\sqrt{\sigma_{xx}(t)}-\sqrt{\bar{\sigma}(t)}\Big)^2,
\end{equation}
$(iii)$ We now provide estimate for $W_2^2(\mu_t,\bar{\mu}_t)$ in the high-friction regime, which corresponds to a large time-scale separation, since the difference $\lambda_1-\lambda_2=\sqrt{\gamma^2-4\omega^2}$ grows with $\gamma$ for fixed $\omega$. We have
\begin{equation}
\label{eq: difference mean}
m(t)-\bar{m}(t)=\frac{e^{-\lambda_2 t}-e^{-\lambda_1 t}}{\lambda_1-\lambda_2}(\lambda_2 x_0+ v_0).
\end{equation}
Therefore, since $|e^{-\lambda_2 t}-e^{-\lambda_1 t}\leq |e^{-\lambda_2 t}|+|e^{-\lambda_1 t}|\leq 2$,
$$
|m(t)-\bar{m}(t)|\leq \frac{2}{\sqrt{\gamma^2-4\omega^2}}\big(\lambda_2 |x_0|+|v_0|\big)\leq \frac{2}{\sqrt{\gamma^2-4\omega^2}}\big(\omega |x_0|+|v_0|\big)
$$
Next we estimate $|\sigma(t)-\bar{\sigma}_t|$. Since
$$
|e^{-\gamma t}-1|\leq e^{-\gamma t} +1\leq 2, \quad |e^{-2\lambda_1 t}-e^{-2\lambda_2 t}|\leq e^{-2\lambda_1 t}+e^{-2\lambda_2 t}\leq 2, \quad \gamma_1\geq \frac{\gamma}{2}
$$
we have
$$
\Big|\frac{4}{\gamma}(e^{-\gamma t}-1)-\frac{e^{-2\lambda_1 t}-e^{-2\lambda_2 t}}{\lambda_1}\Big|\leq \frac{4}{\gamma}|e^{-\gamma t}-1|+\frac{|e^{-2\lambda_1 t}-e^{-2\lambda_2 t}|}{\lambda_1}\leq \frac{12}{\gamma}.
$$
Therefore,
\begin{align}
|\sigma(t)-\bar{\sigma}(t)|&=\left|\frac{1-e^{-2\lambda_2 t}}{\beta\omega^2}\Big[\frac{1}{1-4(\omega/\gamma)^2}-1\Big]+\frac{\gamma\beta^{-1}}{(\gamma^2-4\omega^2)}\Big[\frac{4}{\gamma}(e^{-\gamma t}-1)-\frac{e^{-2\lambda_1 t}-e^{-2\lambda_2 t}}{\lambda_1}\Big]\right|\label{eq: difference variance}
\\&=\left|\frac{1-e^{-2\lambda_2 t}}{\beta}\frac{4(1/\gamma)^2}{1-4(\omega/\gamma)^2}+\frac{\gamma\beta^{-1}}{(\gamma^2-4\omega^2)}\Big[\frac{4}{\gamma}(e^{-\gamma t}-1)-\frac{e^{-2\lambda_1 t}-e^{-2\lambda_2 t}}{\lambda_1}\Big]\right|\notag
\\& \leq \frac{4}{\beta(\gamma^2-4\omega^2)}+\frac{12}{\beta(\gamma^2-4\omega^2)}=\frac{16}{\beta(\gamma^2-4\omega^2)}.\notag
\end{align}
It follows that
\begin{align*}
W_2^2(\mu,\bar{\mu})&=\big(m(t)-\bar{m}(t)\big)^2+\Big(\sqrt{\sigma(t)}-\sqrt{\bar{\sigma}(t)}\Big)^2
\\&\leq \big(m_x(t)-\bar{m}(t)\big)^2+|\sigma_{xx}(t)-\bar{\sigma}(t)|
\\& \leq \frac{4}{\gamma^2-4\omega^2}(\omega |x_0|+|v_0|)^2+\frac{16}{\beta(\gamma^2-4\omega^2)}=\frac{4}{\gamma^2-4\omega^2}\Big[(\omega |x_0|+|v_0|)^2+\frac{4}{\beta}\Big],
\end{align*}
where to obtain the second line from the first line, we have used the inequality $(a-b)^2\leq |a^2-b^2|$ for $a,b\geq 0$.

$(iv)$ We have
\begin{equation*}
   \lim_{t\rightarrow \infty} m(t)=\lim_{t\rightarrow \infty} \bar{m}(t)=0\quad \forall x_0, v_0;\quad \lim_{t\rightarrow \infty}\sigma(t)=\frac{1}{\beta \omega^2}=:\sigma_\infty; \quad\lim_{t\rightarrow \infty}\bar{\sigma}(t)=\frac{1}{\beta\omega^2}=:\bar{\sigma}_\infty=\sigma_\infty.
\end{equation*}
Thus the original dynamics and the reduced one share the same equilibrium measure
$$
\mu_\infty=\bar{\mu}_\infty=\mathcal{N}(0,\sigma_\infty).
$$
Furthermore, we compute the rates of convergence explicitly
\begin{align*}
W_2(\mu_t,\mu_\infty)^2&= (m(t)-m_\infty)^2+(\sqrt{\sigma(t)}-\sqrt{\sigma_\infty})^2
\\&\leq m(t)^2+|\sigma(t)-\sigma_\infty|
\\&=\Big(e^{-\lambda_2 t}x_0+\frac{e^{-\lambda_2 t}-e^{-\lambda_1 t}}{\lambda_1-\lambda_2}(\lambda_2 x_0+ v_0)\Big)^2+\frac{\gamma\beta^{-1}}{(\gamma^2-4\omega^2)}\Big|\frac{4}{\gamma}e^{-\gamma t}-\frac{1}{\lambda_1} e^{-2\lambda_1 t}-\frac{1}{\lambda_2}e^{-2\lambda_2}\Big|
\\&=e^{-2\lambda_2 t}\Big(x_0+\frac{1-e^{-(\lambda_1-\lambda_2)t}}{\lambda_1-\lambda_2}(\lambda_2 x_0+v_0)\Big)^2+\frac{\gamma\beta^{-1}}{(\gamma^2-4\omega^2)}e^{-2\lambda_2t}\Big|\frac{4}{\gamma}e^{-2\lambda_1 t}-\frac{1}{\lambda_1} e^{-2(\lambda_1-\lambda_2) t}-\frac{1}{\lambda_2}\Big|
\\&\leq C e^{-2\lambda_2 t},
\end{align*}
for some constant $C$, which can be computed explicitly (but it is not the focus of this part), where we have used the fact that $\lambda_1>\lambda_2>0$. Thus
$$
W_2(\mu_t,\mu_\infty)\leq C e^{-\lambda_2 t}.
$$
Similarly 
\begin{align*}
 W_2(\bar{\mu}_t,\bar{\mu}_\infty)^2&= (\bar{m}(t)-\bar{m}_\infty)^2+(\sqrt{\bar{\sigma}(t)}-\sqrt{\bar{\sigma}_\infty})^2
\\&\leq \bar{m}(t)^2+|\overline{\sigma}(t)-\overline{\sigma}_\infty|  \\&=e^{-2\lambda_2 t}\Big[x_0^2+\frac{1}{\beta\omega^2}\Big].
\end{align*}
Thus we also obtain 
$$
W_2(\bar{\mu}_t,\bar{\mu}_\infty)\leq C e^{-\lambda_2 t}.
$$

$(v)$ Now we estimate $W_2^2(\mu_t,\bar{\mu}_t)$ in the large time regime. We only need to estimate the difference between the variances $|\sigma(t)-\bar{\sigma}(t)|$. According to \eqref{eq: difference variance}, we have
\begin{align}
\sigma(t)-\bar{\sigma}(t)&=\frac{1-e^{-2\lambda_2 t}}{\beta\omega^2}\Big[\frac{1}{1-4(\omega/\gamma)^2}-1\Big]+\frac{\gamma\beta^{-1}}{(\gamma^2-4\omega^2)}\Big[\frac{4}{\gamma}(e^{-\gamma t}-1)-\frac{e^{-2\lambda_1 t}-e^{-2\lambda_2 t}}{\lambda_1}\Big]
\nonumber\\
&=-\frac{e^{-2\lambda_2 t}}{\beta\omega^2}\Big[\frac{1}{1-4(\omega/\gamma)^2}-1\Big]+\frac{\gamma\beta^{-1}}{(\gamma^2-4\omega^2)}\Big[\frac{4}{\gamma}e^{-\gamma t}-\frac{e^{-2\lambda_1 t}-e^{-2\lambda_2 t}}{\lambda_1}\Big]\nonumber
\end{align}
where, to obtain the second line, we have used the following cancellation
$$
\frac{1}{\beta\omega^2}\Big[\frac{1}{1-4(\omega/\gamma)^2}-1\Big]-\frac{\gamma\beta^{-1}}{(\gamma^2-4\omega^2)}\frac{4}{\gamma}=0.
$$
Therefore, \AM{ it holds }
$$
|\sigma(t)-\bar{\sigma}(t)|\leq \frac{e^{-2\lambda_2 t}}{\beta\omega^2}\Big[\frac{1}{1-4(\omega/\gamma)^2}-1\Big]+\frac{\gamma\beta^{-1}}{(\gamma^2-4\omega^2)}\Big[\frac{4}{\gamma}e^{-\gamma t}+\frac{e^{-2\lambda_2 t}-e^{-2\lambda_1 t}}{\lambda_1}\Big].
$$

$(vi)$ \AM{Now, we can }estimate the Wasserstein distance $W_2^2(\mu_t,\bar{\mu}_t)$ \AM{to explore } the long time behaviour, \AM{viz.}
\begin{align*}
W_2^2(\mu_t,\bar{\mu}_t)&=\big(m(t)-\bar{m}(t)\big)^2+\Big(\sqrt{\sigma(t)}-\sqrt{\bar{\sigma}(t)}\Big)^2
\\&\leq \big(m_x(t)-\bar{m}(t)\big)^2+|\sigma_{xx}(t)-\bar{\sigma}(t)|
\\& \leq \frac{e^{-\lambda_2 t}-e^{-\lambda_1 t}}{\lambda_1-\lambda_2}(\omega |x_0|+ |v_0|)+\frac{e^{-2\lambda_2 t}}{\beta\omega^2}\Big[\frac{1}{1-4(\omega/\gamma)^2}-1\Big]+\frac{\gamma\beta^{-1}}{(\gamma^2-4\omega^2)}\Big[\frac{4}{\gamma}e^{-\gamma t}+\frac{e^{-2\lambda_2 t}-e^{-2\lambda_1 t}}{\lambda_1}\Big]
\\&=\frac{4}{\beta(\gamma^2-4\omega^2)}e^{-\gamma t}+\Big[\frac{\omega|x_0|+|v_0|}{\sqrt{\gamma^2-4\omega^2}}+\frac{4}{\beta(\gamma^2-4\omega^2)}+\frac{\gamma}{\beta\lambda_1(\gamma^2-4\omega^2)}\Big]e^{-\lambda_2 t}
\\&\qquad-\Big[\frac{\omega|x_0|+|v_0|}{\sqrt{\gamma^2-4\omega^2}}+\frac{\gamma}{\beta\lambda_1(\gamma^2-4\omega^2)}\Big]e^{-\lambda_1 t}
\\& \leq \Big[\frac{\omega|x_0|+|v_0|}{\sqrt{\gamma^2-4\omega^2}}+\frac{8}{\beta(\gamma^2-4\omega^2)}+\frac{\gamma}{\beta\lambda_1(\gamma^2-4\omega^2)}\Big]e^{-\lambda_2 t}
\\& \leq \Big[\frac{\omega|x_0|+|v_0|}{\sqrt{\gamma^2-4\omega^2}}+\frac{10}{\beta(\gamma^2-4\omega^2)}\Big]e^{-\lambda_2 t}.
\end{align*}
\AM{Here} we have used the fact that $\gamma\geq \lambda_2$ and $\frac{\gamma}{\lambda_1}=\frac{2\gamma}{\gamma+\sqrt{\gamma^2-4\omega^2}}\leq 2$.

$(vi)$. Suppose that $x_0$ is randomly distributed following an even distribution $\rho_0$. Then the laws of $x(t)$ and $\bar{x}(t)$ are given by
$$
\mu_t=\mathcal{N}(m(t),\sigma(t))\ast \rho_0,\quad \mu_t=\mathcal{N}(\bar{m}(t),\bar{\sigma}(t))\ast \rho_0.
$$
Since $\mathcal{N}(m(t),\sigma(t)),\mathcal{N}(\bar{m}(t),\bar{\sigma}(t))\in \mathcal{P}_2(\mathbb{R})$, according to \cite[Lemma 5.2]{santambrogio2015optimal} we have
$$
W_2^2(\mu_t,\bar{\mu}_t)=W_2^2(\mathcal{N}(m(t),\sigma(t))\ast \rho_0, \mathcal{N}(\bar{m}(t),\bar{\sigma}(t))\ast \rho_0)\leq W^2(\mathcal{N}(m(t),\sigma(t)),\mathcal{N}(\bar{m}(t),\bar{\sigma}(t)),
$$
thus the upper bound estimates in the two previous parts are still true.
\end{proof}

\section{Model reduction of two coupled underdamped Brownian oscillators}
\label{sec:sec4}

We now proceed with the computation of the Wasserstein distance for a slightly more elaborate model, corresponding to a system of two coupled overdamped Brownian harmonic oscillators. The dynamics of the model can conveniently be written as follows:
\begin{subequations}
\label{eq: coupled oscillator}
\begin{align}
    \dot{x}_1&=a x_1+k (x_2-x_1)+\sigma_1 \dot{W}_1\\
    \dot{x}_2&=-k (x_2-x_1)+d x_2+\sigma_2 \dot{W}_2,
\end{align}
\end{subequations}
where $\dot{W}$ denotes the formal derivative of a Wiener process, corresponding to a white noise, $a, d<0$ are parameters characteristic of the individual oscillator (without loss of generality we also assume $a\ge d$), $\sigma_1,\sigma_2>0$ denote the noise strenghts, \AM{and finally}, $k>0$ is the coupling parameter.

The system \eqref{eq: coupled oscillator} represents the overdamped version of the coupled underdamped Langevin dynamics of the two oscillators. A contracted description for the deterministic case (i.e., with $\sigma_1=\sigma_2=0$) under a suitable assumption of scale separation is studied, with applications to relaxation dynamics in proteins, in \cite{Soheilifard2011}. We can derive a reduced system by eliminating the variable $x_2$, in \eqref{eq: coupled oscillator}, using the procedure introduced in \cite{CM22,CDM22}. This consists of two distinct steps: \AM{(i)} the
deterministic component of the dynamics is obtained using the Invariant Manifold method, then \AM{(ii)} the diffusion terms are determined
via \AM{fulfilling} the Fluctuation-Dissipation relation. 

\subsection{Deterministic evolution}
Let $\langle \mathcal{O}\rangle$ denote the average over noise of the variable $\mathcal{O}$. The original dynamics can be written as
\begin{equation}
\dot{\mathbf{z}}=\mathbf{Q}~ \mathbf{z} \;, \label{orig}
\end{equation}
where $\mathbf{z}=(\langle x_1\rangle, \langle x_2\rangle)$ and
\begin{equation}
\label{eq: Q}
\mathbf{Q}=\mathbf{Q}(k)=\begin{pmatrix}
a-k&k\\ k &-k+d
\end{pmatrix}    
\end{equation}

The characteristic polynomial of $\mathbf{Q}$ is
$$
\lambda^2-(a+d-2k)\lambda+(ad-ak-dk)=0.
$$
Thus $\mathbf{Q}$ has two real negative eigenvalues:
\begin{equation}
\label{eq: lambda}
\lambda_{\pm}=\lambda_{\pm}(k):=\frac{(a+d-2k)\pm \sqrt{(a-d)^2+4k^2}}{2} \;,
\end{equation}

\AM{In this model, the time-scale separation is encoded in}  the difference $\lambda_+-\lambda_-=\sqrt{(a-d)^2+4k^2}$, which grows with \AM{increasing} $k$, for fixed parameters $a,d$.
We seek a closure of the form
$\langle x_2\rangle=\alpha \langle x_1\rangle$, hence, following \cite{CDM22}, we define a \textit{macroscopic} time derivative of $\langle x_2 \rangle$ via the chain rule:
\begin{align*}
\partial_t^{macro} \langle x_2 \rangle&:=\frac{\partial \langle x_2\rangle}{\partial \langle x_1\rangle}\langle\dot{x}_1\rangle
\\&=
(\alpha (a-k)+ \alpha^2 k)\langle x_1\rangle \;,
\end{align*}  
which expresses the \textit{slaving principle} mentioned in Sec. \ref{sec:sec1}. Furthermore, we also define the \textit{microscopic} time derivative of $\langle x_2 \rangle$ in terms of the vector field given in Eq. \eqref{orig}, where $\langle x_2\rangle$ is expressed through the aforementioned closure. We thus set:
\begin{align*}
\partial_t^{micro} \langle x_2 \rangle&:=k \langle x_1\rangle+(d-k)\langle x_2\rangle
\\
&=(k+\alpha (d-k))\langle x_1\rangle \;.
\end{align*}

The Invariant Manifold method requires that microscopic and macroscopic time derivatives of $\langle x_2\rangle$ coincide, independently of the values of the observable $\langle x_1\rangle$. Thus, we obtain the following invariance equation

\begin{equation}
    \label{eq: Invariance equation}
    \alpha (a-k)+\alpha^2 k=k+\alpha (d-k)\quad \Longleftrightarrow \quad k\alpha^2+(a-d)\alpha +k=0 \;,
\end{equation}
which has two solutions
$$
\alpha_{\pm}=\alpha_{\pm}(k):=\frac{-(a-d)\pm \sqrt{(a-d)^2+4k^2}}{2k} \;.
$$
The reduced dynamics for the deterministic part is
\begin{equation}
    \label{eq: reduced dynamics}
    \langle \dot{x}_1\rangle =(a-k+ k \hat{\alpha})\langle x_1\rangle \;,
\end{equation}
where $\hat{\alpha}\in\{\alpha_+, \alpha_-\}$ which will be specified later. It is noticeable that
$$
a-k+k\alpha_{\pm}=\frac{(a+d-2k)\pm\sqrt{(a-d)^2+4k^2}}{2}\equiv \lambda_{\pm} \;.
$$
\AM{Looking at} \eqref{eq: reduced dynamics}, \AM{ we notice that } the coefficient multiplying $\langle x_1 \rangle$  coincides with one of the eigenvalues of the matrix $\mathbf{Q}$. To pick up the right eigenvalue, we use the following criterion.  
We select $\hat{\alpha}$ from solutions $\alpha_\pm$ to the invariance equation \eqref{eq: Invariance equation} that satisfies $a-k+k \hat{\alpha} \rightarrow a<0$ as $k\rightarrow 0$, that is $k\hat{\alpha} \rightarrow 0$ as $k\rightarrow 0$. Since we assume that $a\geq d$, we take
$$
\hat{\alpha}=\alpha_{+}=\frac{-(a-d)+ \sqrt{(a-d)^2-4k^2}}{2k} \;.
$$

\subsection{Incorporating the noise}
To characterize the noise term, we \AM{employ the methodology proposed in } \cite{CDM22}. Therefore, we first define the diffusion matrix $\boldsymbol{D}$ as
\begin{equation}
\label{eq: D}
\boldsymbol{D}=\begin{pmatrix}
    \sigma_1& 0\\ 0& \sigma_2
\end{pmatrix} \; ,
\end{equation}
and we also denote 
$$
\dot{\mathbf{W}}=(\dot{W}_1,\dot{W}_2) \; .
$$
The solution of Eqs. \eqref{eq: coupled oscillator} reads:
\begin{equation}
\mathbf{z}(t)=e^{\mathbf{Q}t}\mathbf{z}_0+\int_0^t e^{\mathbf{Q}(t-s)}\mathbf{D}~\dot{\mathbf{W}} ds \; .
\end{equation}
We thus find
\begin{align*}
   \lim_{t\rightarrow\infty} \mathbb{E}[z_1(t)^2]\equiv\overline{\Sigma}_{11}= -\frac{1}{2}\Big(
\frac{1}{a-2k}+\frac{1}{a}\Big).
\end{align*}
The full reduced system takes hence the form
\begin{equation}
d\hat{x}(t)=\lambda_+\hat{x}(t)\,dt + \hat{D}dW_t \;, \label{redeq}
\end{equation}
where the drift coefficient $\lambda_+$ is defined in \eqref{eq: lambda} and the diffusion coefficient $\hat{D}$ is given by
\begin{equation}
\hat{D}=-\lambda_+\overline{\Sigma}_{11} \;. \label{redeq2}
\end{equation}

\subsection{Quantification of errors and \AM{the} long-time behaviour}
In this section we will compute explicitly the Wasserstein distance between the laws of the original dynamics of $x_1$ and of the reduced dynamics \eqref{redeq} and study their long-time behaviour. The Fokker Planck equation associated to the full original dynamics \eqref{eq: coupled oscillator} is given by the following linear-drift diffusion equation
\begin{equation}
\label{eq: PDE}
   \partial_t\rho=-\div(\mathbf{Q}\rho)+\div(\boldsymbol{D}\nabla \rho), 
\end{equation}
where $\rho=\rho(t,x_1,x_2)$ is the joint probability density of $(x_1, x_2)$, the drift matrix $\boldsymbol{Q}$ and the diffusion matrix $\boldsymbol{D}$ are given in \eqref{eq: Q} and \eqref{eq: D} respectively. Note that the above system is a special case of the general drift-diffusion equation introduced in Section \ref{sec: linear SDE}.

Since we are focusing on the role of the coupling parameter, for simplicity of presentation, we consider identical oscillator, that is $a=d<0$ and normalising $\sigma_1=\sigma_2=1$, so that
$$
\mathbf{Q}=\begin{pmatrix}
    a-k&k\\ k&a-k
\end{pmatrix},\quad\text{and}\quad \boldsymbol{D}=I.
$$
The main result of this section is the following theorem.
\begin{theorem} Let $\rho_1(t)$ be the distribution of $x_1(t)$ of the original coupled dynamics \eqref{redeq} starting at a deterministic initial data $(x_1,x_2)(0)=(x_1,x_2)$, and $\hat{\rho}_1(t)$ be the distribution of the reduced dynamics \eqref{redeq} starting from $x_1$. Then there exists a constant $C>0$ such that the following statements hold
   \begin{enumerate}[(i)]
       \item $W_2(\rho_1(t),\hat{\rho}_1(t))^2\leq C k$.
       \item $\max\{W_2(\rho_1(t),\rho_\infty), W_2(\hat{\rho}_1(t),\rho_\infty)\}\leq C e^{at}$, where $\rho_\infty=\mathcal{N}(0,\overline{\Sigma}_{11})$
       \item $W_2(\rho_1(t),\hat{\rho}_1(t)\leq C e^{at}$.
   \end{enumerate}
\end{theorem}
\begin{proof}
According to Lemma \ref{lem: Gaussian sol}, the solution to \eqref{eq: PDE} is given by
$
\rho(t,x_1,x_2)=\mathcal{N}(\mu(t), \Sigma(t)),
$
where 
$$
\mu(t)=e^{t\mathbf{Q}}\begin{pmatrix}
    x_1\\x_2
\end{pmatrix},\quad \Sigma(t)=2\int_0^t e^{s\mathbf{Q}} e^{s\mathbf{Q}^T}\,ds.
$$
Since $\mathbf{Q}\mathbf{Q}^T=\mathbf{Q}^T\mathbf{Q}$ and $\mathbf{Q}=\mathbf{Q}^T$, we have
$$
e^{s\mathbf{Q}} e^{s\mathbf{Q}^T}=e^{s(\mathbf{Q}+\mathbf{Q}^T)}=
e^{2s\mathbf{Q}}.
$$
Thus, we can simplify $\Sigma(t)$ as
\begin{equation*}
\Sigma(t)=2\int_0^t e^{2s\mathbf{Q}}\,ds.
\end{equation*}
Applying lemma \ref{lem: exponential matrix}, we compute 
\begin{align*}
 e^{t\mathbf{Q}}&= \frac{1}{\Delta}\begin{pmatrix}
    m_{11}& m_{12}\\ m_{21}& m_{22}
 \end{pmatrix}, \quad \Delta= 2kt,\\
 m_{11}&=m_{22}=e^{(a-k) t}\Delta \cosh{\frac{1}{2}\Delta}=\frac{1}{2}e^{(a-k) t}\Delta(e^{kt}+e^{-kt})=\frac{1}{2}\Delta (e^{(a-2k)t}+e^{at})\\
 m_{12}&=m_{21}=2kt e^{(a-k) t} \sinh{\frac{1}{2}\Delta}=\frac{1}{2}\Delta e^{(a-k)t}(e^{kt} -e^{-kt}).=\frac{1}{2}\Delta (e^{at} -e^{(a-2k)t}).
\end{align*}
Thus
$$
e^{t\mathbf{Q}}=\frac{1}{2}\begin{pmatrix}
e^{(a-2k)t}+e^{at}& e^{at} -e^{(a-2k)t}\\
e^{at} -e^{(a-2k)t}&e^{(a-2k)t}+e^{at}
\end{pmatrix}.
$$
Similarly
$$
e^{2t\mathbf{Q}}=\frac{1}{2}\begin{pmatrix}
e^{2(a-2k)t}+e^{2at}& e^{2at} -e^{2(a-2k)t}\\
e^{2at} -e^{2(a-2k)t}&e^{2(a-2k)t}+e^{2at}
\end{pmatrix}.
$$
Therefore,
$$
\Sigma(t)=2\int_0^t e^{2s Q}\,ds=\frac{1}{2}\begin{pmatrix}
 \frac{e^{2(a-2k)t}-1}{a-2k}+\frac{e^{2at}-1}{a}&\frac{e^{2at}-1}{a}- \frac{e^{2(a-2k)t}-1}{a-2k}  \\
 \frac{e^{2at}-1}{a}- \frac{e^{2(a-2k)t}-1}{a-2k}   &\frac{e^{2(a-2k)t}-1}{a-2k}+\frac{e^{2at}-1}{a}
\end{pmatrix}.
$$
It follows that   
$$
\rho_1(t)=\mathcal{N}(\mu_1(t),\Sigma_{11}(t))=\mathcal{N}\Bigg(\frac{1}{2}\Big((e^{(a-2k)t}+e^{at})x_1+( e^{at} -e^{(a-2k)t})x_2\Big),\frac{1}{2}\Big(
 \frac{e^{2(a-2k)t}-1}{a-2k}+\frac{e^{2at}-1}{a}\Big)\Bigg),
$$
Since $\hat{x}$ is an OU process, we obtain
$$
\hat{\rho}_1(t)=\mathcal{N}(\hat{\mu}_1(t),\hat{\Sigma}_{1}(t))=\mathcal{N}\Big(e^{\lambda_+ t}x_1,-\frac{\hat{D}}{\lambda_+}(1-e^{2\lambda_+ t})\Big)=\mathcal{N}\Big(e^{\lambda_+ t}x_1,\overline{\Sigma}_{11}(1-e^{2\lambda_+ t})\Big),
$$
recalling that, with $a=d$
$$
\lambda_+=\frac{(a+d-2k)+\sqrt{(a-d)^2+4k^2}}{2}=a,\quad \overline{\Sigma}_{11}=-\frac{1}{2}\Big(\frac{1}{a-2k}+\frac{1}{a}\Big).
$$
The Wasserstein distance between $\rho_1$ and $\hat{\rho}_1$ is given by
\begin{equation}
    W_2(\rho_1(t),\hat{\rho}_1(t))^2=(\mu_1(t)-\hat{\mu}_1(t))^2+\Big(\sqrt{\Sigma_{11}}(t)-\sqrt{\hat{\Sigma}_{1}}(t)\Big)^2
\end{equation}
$(i)$ We compute
\begin{align}
|\mu_1(t)-\hat{\mu}_1(t)|&=\frac{1}{2}\Big|\Big(e^{(a-2k)t}+e^{at})x_1+( e^{at} -e^{(a-2k)t})x_2\Big)-e^{at}x_1\Big|\notag
\\&=\frac{1}{2}|(e^{at}-e^{(a-2k)t})(x_2-x_1)|\notag
 \\&=  \frac{1}{2}e^{at}|x_2-x_1|(1-e^{-2kt})\label{mean}
 \\& \leq |x_2-x_1|k e^{at}t \notag
 \\& \leq k\,|x_2-x_1| \frac{1}{|a| e},\notag
\end{align}
where in the first inequality we have used the elementary inequality $1-e^{-x}\leq x$ for all $x>0$, and in the last inequality we have used (noting that $a<0$)
\begin{equation}
\label{eq:ineq}
\max_{t>0}t e^{at}=\frac{1}{|a| e}.
\end{equation}
We also estimate
\begin{align}
\Sigma_{11}(t)-\hat{\Sigma}_{1}(t)&=\frac{1}{2}\Bigg[\frac{1}{a-2k}\Big(e^{2(a-2k)t}-e^{2\lambda_+ t}\Big)+\frac{1}{a}\Big(e^{2at}-e^{2\lambda_+ t}\Big)\Bigg]\notag
\\&=\frac{1}{2}\frac{1}{a-2k}\Big(e^{2(a-2k)t}-e^{2a t}\Big)\notag
\\&=\frac{1}{2}\frac{1}{2k-a}e^{2at}\Big(1-e^{-4kt}\Big)\label{diff}
\\&\leq \frac{2k}{2k-a}e^{2at}t\notag
\\& \leq \frac{k}{a^2 e},\notag
\end{align}
where to go from \eqref{diff} to the next line, we have used $1-e^{-4kt}\leq 4kt$ and \eqref{eq:ineq} again (with $a$ replaced by $2a$).
Therefore, we have
\begin{align*}
   W_2(\rho_1(t),\hat{\rho}_1(t))^2&\leq (\mu_1(t)-\hat{\mu}_1(t))^2+\Big|\Sigma_{11}(t)-\overline{\Sigma}_{1}(t)\Big| 
   \\& \leq  k^2\,|x_2-x_1|^2 \frac{1}{|a|^2 e^2}+ \frac{k}{a^2 e}
   \leq C k,
\end{align*}
for any bounded $k$. 

$(iii)$
Since $a<0$,
$$
\lim_{t\rightarrow\infty}\mu_1(t)=\lim_{t\rightarrow \infty}\hat{\mu}_1(t)=0,\quad
\lim_{t\rightarrow\infty}\Sigma_{11}(t)=-\frac{1}{2}\Big(\frac{1}{a-2k}+\frac{1}{a}\Big)=\overline{\Sigma}_{11}.
$$
it implies that
$$
\lim_{t\rightarrow \infty} \rho_1(t)=\lim_{t\rightarrow \infty} \hat{\rho}_1(t)=\rho_\infty=\mathcal{N}(0,\overline{\Sigma}_{11}).
$$
We can also compute explicitly the rates of convergence of these limits in the Wasserstein distance. We have
\begin{align}
\label{1}
 W_2(\rho_1(t),\rho_\infty)^2=\mu_1(t)^2+\Big(\sqrt{\Sigma_{11}}(t)-\sqrt{\overline{\Sigma}}\Big)^2   \leq \mu_1(t)^2+\Big|\Sigma_{11}(t)-\overline{\Sigma}_{11}\Big|.
\end{align}
We estimate each term on the right hand side of \eqref{1}. For the first term, we get
\begin{align}
\label{2}
 \mu_1(t) &=\frac{1}{2}\Big((e^{(a-2k)t}+e^{at})x_1+( e^{at} -e^{(a-2k)t})x_2\Big)\notag
 \\&=\frac{1}{2}e^{at}\Big((1+e^{-2kt})x_1+(1-e^{-2kt})x_2\Big)\notag
 \\& \leq C e^{at}.
\end{align}
For the second term, we have
\begin{align}
\label{3}
 |\Sigma_{11}(t)-\overline{\Sigma}_{11}|=\frac{1}{2}\Big|\frac{e^{2(a-2k)t}}{a-2k}+\frac{e^{2at}}{a}\Big|=\frac{1}{2}e^{2at}\Big|\frac{1}{a}+\frac{e^{-4kt}}{a-2k}\Big|\leq C e^{2at}.   
\end{align}
Substituting \eqref{2} and \eqref{3} to \eqref{1}, we obtain
$$
W_2(\rho_1(t),\rho_\infty)\leq C e^{at},
$$
thus $\rho_1$ exponentially converges, with a rate $a$, to $\rho_\infty$. Similarly,
\begin{align*}
 W_2(\hat{\rho}_1(t),\rho_\infty)^2&=\hat{\rho}_1(t)^2+\Big(\sqrt{\hat{\Sigma}_1}(t)-\sqrt{\overline{\Sigma}}_{11}\Big)^2
 \\&\leq \hat{\rho}_1(t)^2+\Big|\hat{\Sigma}_1(t)-\overline{\Sigma}_{11}\Big|
 \\& =(x_1^2 +\overline{\Sigma}_{11}) e^{2at}.
\end{align*}
Hence $\hat{\rho}_1$ exponentially converges with the same rate  $a$ to $\rho_\infty$.

$(iv)$ According to \eqref{mean} and \eqref{diff} we have
\begin{align*}
&|\mu_1(t)-\hat{\mu}_1(t)|=  \frac{1}{2}e^{at}|x_2-x_1|(1-e^{-2kt})\leq C e^{at},
\\& |\Sigma_{11}(t)-\hat{\Sigma}_{1}(t)|=\frac{1}{2}\frac{1}{|2k-a|}e^{2at}\Big(1-e^{-4kt}\Big)\leq C e^{2at}.
\end{align*}
Thus
\begin{equation*}
 W_2(\rho_1(t),\hat{\rho}_1(t))^2\leq (\mu_1(t)-\hat{\mu}_1(t))^2+|\Sigma_{11}(t)-\hat{\Sigma}_{1}(t)|\leq C e^{2at},
\end{equation*}
that is 
$$
W_2(\rho_1(t),\hat{\rho}_1(t))\leq C e^{at}.$$
This completes the proof of this theorem. We remark that we have assumed deterministic initial data, but the theorem can also be extended to the case where the initial data follow symmetric distributions as in Section \ref{sec:sec3}.
\end{proof}

\section{Summary and outlook}
\label{sec:sec5}
In this work we have employed the reduction scheme recently introduced in \cite{CM22,CDM22}, \AM{which} suitably combines the Invariant Manifold method \AM{ with} the Fluctuation-Dissipation relation, to derive a contracted description for two classical models of statistical physics, namely  the underdamped Brownian harmonic oscillator and a system of two coupled overdamped Brownian harmonic oscillators. The present work significantly extends the previous results: we succeeded here to quantify explicitly the error between the original and the reduced dynamics, as well as their rates of convergence to equilibrium. The technical tool we used is the Wasserstein distance, which is widely employed in the theory of optimal transport. We have thus shown that the two dynamics are exponentially close at any time, share the same equilibrium measure, and exponentially converge to the \AM{same} equilibrium measure with the same rate. Furthermore, \AM{the two dynamics} are also found to coincide if the relevant parameter controlling the time-scale separation of the \AM{original} model is sent to infinity. The linearity of the considered models has clearly played an important role in the analysis of this work, enabling the explicit computations of their solutions and of the involved Wasserstein distances. A key challenge for future developments is to generalize \AM{our analysis in order to deal with} non-linear models, where explicit solutions and computations are not accessible. Another direction of research points toward the investigations of systems with a large numbers of degrees of freedom, e.g. models relevant to climate dynamics \cite{hummel2023reduction}, or small systems of interest in modern nanotechnologies, such as biomolecular motors \cite{Tang16}.

\section*{Acknowledgements}
MC's research was performed under
the auspices of Italian National Group of Mathematical Physics (GNFM) of INdAM. MHD research was supported by EPSRC grants EP/W008041/1 and EP/V038516/1.

\bibliographystyle{alpha}
\bibliography{biblio}
\end{document}